\documentclass{amsart}

\usepackage{pgfplots}
\pgfplotsset{compat=1.12}

\usepackage[active]{srcltx}

\usepackage{graphicx}

\usepackage{latexsym}
\usepackage{amsmath,amsthm}
\usepackage{amsfonts}
\usepackage[psamsfonts]{amssymb}
\usepackage{enumerate}

\usepackage[final]{showkeys}
\usepackage{url}

\RequirePackage[OT1]{fontenc}

\usepackage{calrsfs}

\usepackage{hyperref}

\theoremstyle{plain} 
\newtheorem{theorem}{Theorem}[section]
\newtheorem{corollary}[theorem]{Corollary}
\newtheorem{lemma}[theorem]
{Lemma}
\newtheorem{proposition}[theorem]{Proposition}

\theoremstyle{definition} 

\theoremstyle{definition} 

\theoremstyle{remark} 

\theoremstyle{remark} 

\newtheorem*{remark*}{Remark}
\numberwithin{equation}{section}

\makeatletter
  \renewcommand\section{\@startsection {section}{1}{\z@}%
                                   {-\bigskipamount}%
                                   {\medskipamount}%
                                   {\large\bfseries
                                   \raggedright}}

  \renewcommand\subsection{\@startsection {subsection}{2}{\z@}%
                                   {-\medskipamount}%
                                   {\smallskipamount}%
                                   {\bfseries
                                   \raggedright}}
\makeatother

\newcommand{\sign}{\operatorname{sign}}
\newcommand{\card}{\operatorname{card}}

\newcommand{\arccot}{\operatorname{arccot}}

\renewcommand{\gg}{>\kern-2pt>}
\renewcommand{\ll}{<\kern-2pt<}

\renewcommand{\le}{\leqslant}
\renewcommand{\ge}{\geqslant}

\newcommand{\vv}{\mathbf{v}}
\newcommand{\w}{\mathbf{w}}

\renewcommand{\th}{\theta}

\newcommand{\si}{\sigma}

\newcommand{\de}{\delta}

\newcommand{\LD}{\mathcal{L}\!\mathcal{D}}
\renewcommand{\LD}{\mathcal{L}{\kern -1.9pt}\mathcal{D}}
\renewcommand{\LD}{\mathcal{D}}
\renewcommand{\LD}{\mathcal{L}{\kern -4pt}\mathcal{C}}
\renewcommand{\LD}{\mathcal{R}{\kern -3pt}\mathcal{C}}

\newcommand{\ii}[1]{\mathrm{I}\!\left\{#1\right\}}

\newcommand{\R}{{\mathbb{R}}}

\newcommand{\vp}{\varepsilon}

\begin{document}


\title{A simpler $\R^3$ realization of the M\"obius strip}


\author{Iosif Pinelis
}

\address{Department of Mathematical Sciences\\
Michigan Technological University\\
Hough\-ton, Michigan 49931, USA\\
E-mail: ipinelis@mtu.edu}



%
%
%
%

\begin{abstract}
A very simple $\R^3$ realization of the M\"obius strip, significantly simpler than the common one, is given. For any, however large width/length ratio of the strip, it is shown that this realization, in contrast with the common one, is the 
union of a vertical segment 
and the graph of a simple rational function on a simple polynomially defined subset of $\R^2$.  
\end{abstract}

\maketitle

\tableofcontents

\section{Summary and discussion}\label{intro} 
Take any positive real number $\de$. 
The M\"obius strip (of width $2\de$) can be defined as the topological space on the set 
\begin{equation*}
P_\de:=[0,2\pi)\times[-\de,\de] 	
\end{equation*}
with the topology generated by the base consisting of all sets of the form \break 
$P_\de\cap\big(B_{p,\vp}\cup B_{\tilde p,\vp}\big)$, where $p=(t,r)\in\R^2$, $\tilde p
:=(t+2\pi,-r)$, $\vp\in(0,\infty)$, and 
$B_{p,\vp}$ is the open ball in $\R^2$ of radius $\vp$ centered at the point $p$. 

Informally, the M\"obius strip is obtained by taking the rectangle $\overline P_\de=[0,2\pi]\times[-\de,\de]$ with its usual topology, and then identifying each point of the form $(2\pi,r)\in\overline P_\de$ 
with the point $(0,-r)$. Of course, a topologically equivalent object will be obtained using, instead of $[0,2\pi]\times[-\de,\de]$, the product of any two closed intervals of finite nonzero lengths. 
However, the choice of the rectangle made here provides for slightly simpler notation. 

The M\"obius strip may be realized in (that is, homeomorphically embedded into) $\R^3$ 
in a variety of ways; here and elsewhere in this paper, the topology on any subset of $\R^3$ is assumed to be induced by the standard topology on $\R^3$. Perhaps the most common $\R^3$ realization of the M\"obius strip is via parametric equations 
\begin{equation}\label{eq:old}
(x,y,z)=\vv(t,r):=\big(\cos t+r\cos\tfrac t2\,\cos t,\sin t+r\cos\tfrac t2\,\sin t,r\sin\tfrac t2\big)  
\end{equation}
for $(t,r)\in P_\de$. 
Let us refer to this as \emph{the common (M\"obius strip) realization}. 
It can be visualized as follows. We take the closed segment of length $2\de$ along the $x$-axis centered at point $1$ on this axis; then the plane containing this segment and the $z$-axis is rotated about the $z$-axis; simultaneously, the segment is rotated at half the angular speed about its center in the revolving plane. 
Thus, the trajectory of the center of the moving segment is the circle of radius $1$ in the $xy$-plane centered at the origin of that plane, and so, one may refer to this surface as the common M\"obius strip realization \emph{of radius $1$}; by homothety, one can easily define the common M\"obius strip realization of any positive real radius. 

 
If $\de$ is not too large, then the resulting set -- traced out by the moving segment -- 
will be a realization of the M\"obius strip, that is, its homeomorphic image under the map given by parametric equations \eqref{eq:old}. 

In this note, we offer a simpler realization of the M\"obius strip -- given by parametric equations 
\begin{equation}\label{eq:new}
(x,y,z)=\w(t,r):=\big(\cos t+r\cos\tfrac t2,\sin t+r\sin\tfrac t2,r\sin\tfrac t2\big) 
\end{equation}
for $(t,r)\in P_\de$. So, instead of the products $\cos\tfrac t2\,\cos t$ and $\cos\tfrac t2\,\sin t$ in \eqref{eq:old}, we now have just $\cos\tfrac t2$ and $\sin\tfrac t2$. 
Here, instead of the segment revolving in the revolving plane through the $z$-axis, we have a moving segment, with endpoints $\w(t,-\de)$ and $\w(t,\de)$ at ``time'' $t$,  
which remains parallel to the plane $y=z$. Indeed, the last two coordinates of the vector $\w(t,\de)-\w(t,-\de)=2\de\,\big(\cos\tfrac t2,\sin\tfrac t2,\sin\tfrac t2\big)$ are equal to each other. 
The center of the moving segment still remains on the unit circle in the $xy$-plane. 
So, formula \eqref{eq:new} defines what may be referred to as \emph{the simple M\"obius strip realization (of radius $1$)}; again, by homothety, one can easily define the simple M\"obius strip realization of any positive real radius. 

These two realizations, the common \eqref{eq:old} and simple \eqref{eq:new} ones, of the M\"obius strip are shown in Figure~1, 
where, in particular, one can see mentioned ``moving'' segments. 
In this note, we allow ourselves some abuse of terminology, letting the term ``realization'' refer both to one of the maps $\vv$ or $\w$ and to the corresponding image $\vv(P_\de)$ or $\w(P_\de)$. 

\begin{figure}[h]
	\centering
		\includegraphics[width=.90\textwidth]{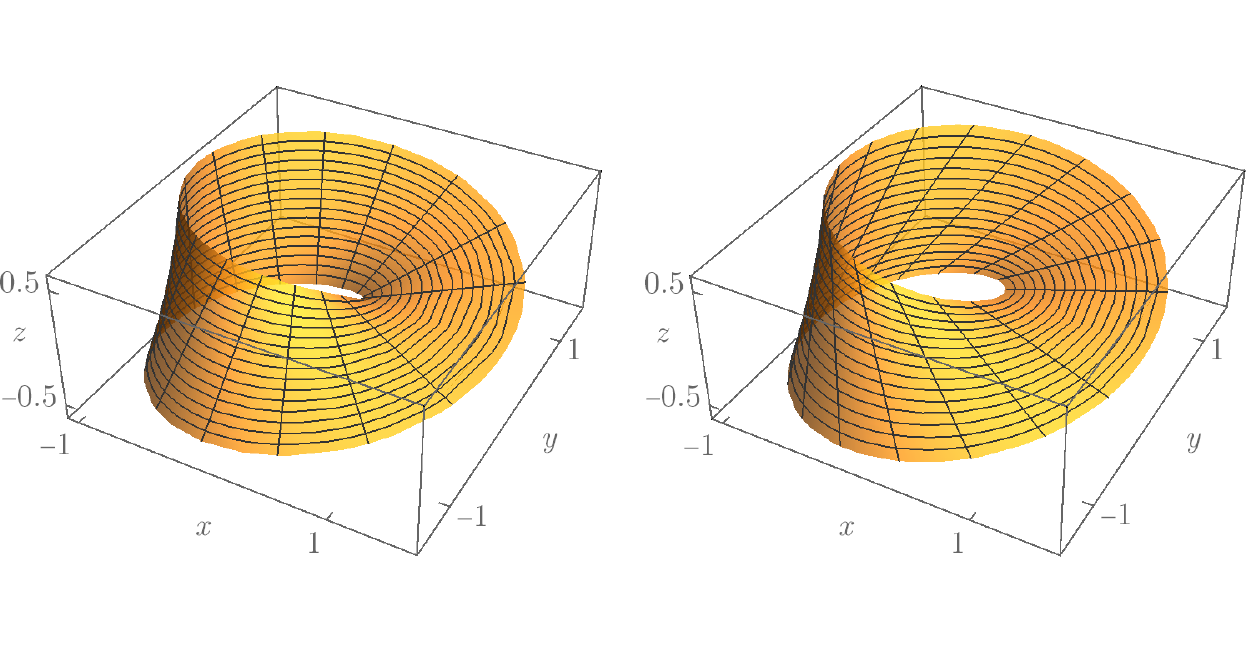}
	\label{fig:1}
	\caption{Two realizations of the M\"obius strip: the common one (left panel) and the simple one (right panel), both with 
$\de=3/5$.}  
\end{figure}


The``moving segment'' heuristics concerning the simple realization of the M\"obius strip is confirmed and detailed in 


\begin{theorem}\label{th:new}
The map $P_\de\ni(t,r)\mapsto\w(t,r)\in\w(P_\de)$ is a homeomorphism 
iff $\de\le\sqrt2$. 
\end{theorem}

For any point $p=(t,r)\in P_\de$, we use $\w(p)$ and $\w(t,r)$ interchangeably. 

The necessary proofs are deferred to Section~\ref{proofs}. 

Introduce the following notation:  
\begin{gather}
	Q_\de:=\big\{(p_1,p_2)\in P_\de^2\colon\   
	p_1\ne p_2, \  \w(p_1)=\w(p_2)\big\} \label{eq:Q}
\\
I_\de:=\{\w(p_1)\colon\  \exists p_2\ \, (p_1,p_2)\in Q_\de\}, \label{eq:I}
\end{gather} 
\begin{equation}\label{eq:c,s}
	c_\de:=1\wedge\frac{\de\vee\sqrt2}2,\quad s_\de:=2c_\de\sqrt{1-c_\de^2}. 
\end{equation}
So, the set $Q_\de$ is 
the set of all pairs of distinct points in $P_\de$ that are glued together by the map $\w$, whereas $I_\de$ is the image under $\w$ of the set of all points in $P_\de$ that are glued (by the map $\w$) together with a different point in $P_\de$. Thus, $I_\de$ may be referred to as the self-intersection set of the image $\w(P_\de)$ of $P_\de$ under the map $\w$.  

The proof of Theorem~\ref{th:new} is mainly based on the following proposition, which provides a detailed description of self-intersection properties of the simple M\"obius strip realization. 

\begin{proposition}\label{prop:new}
Take any points $p_1=(t_1,r_1)$ and $p_2=(t_2,r_2)$ in 
$[0,2\pi)\times\R$. Then $(p_1,p_2)\in Q_\de$ iff for some $k\in\{0,1\}$ one has 
\begin{equation}\label{eq:conds}
	t_2=(2k+1)\pi-t_1,\quad r_1=-2\cos\tfrac{t_1}2,\quad r_2=(-1)^{k+1}2\sin\tfrac{t_1}2,\quad 
	s_\de\le|\sin t_1|<1.  
\end{equation}
Moreover, 
\begin{equation}\label{eq:I=}
	I_\de=\{(-1,0,s)\colon s_\de\le|s|<1\}. 
\end{equation}
\end{proposition}

\medskip

In view of the definitions in \eqref{eq:c,s}, we always have $s_\de\le1$, and 
$s_\de=1$ iff $\de\le\sqrt2$. So, by \eqref{eq:I=}, for the self-intersection set $I_\de$ we have this: 
\begin{equation}\label{eq:I=empty}
	I_\de=\emptyset\iff\de\le\sqrt2. 
\end{equation}
One may also note at this point that, if $\de\ge2$, then $s_\de=0$ and hence the self-intersection set $I_\de$ in this case is the entire open interval between the points $(-1,0,-1)$ and $(-1,0,1)$ in $\R^3$.

Viewing the entire image $\w(P_\de)$ of the M\"obius strip, 
it is a bit problematic to see how it looks like  
near the self-intersection set $I_\de$, because that is covered by folds of the surface. 
For instance, see Figure~2, 
which is the version of Figure~1 
with $\de=1.97$ instead of $\de=3/5$. 
\begin{figure}[h]
	\centering
		\includegraphics[width=.90\textwidth]{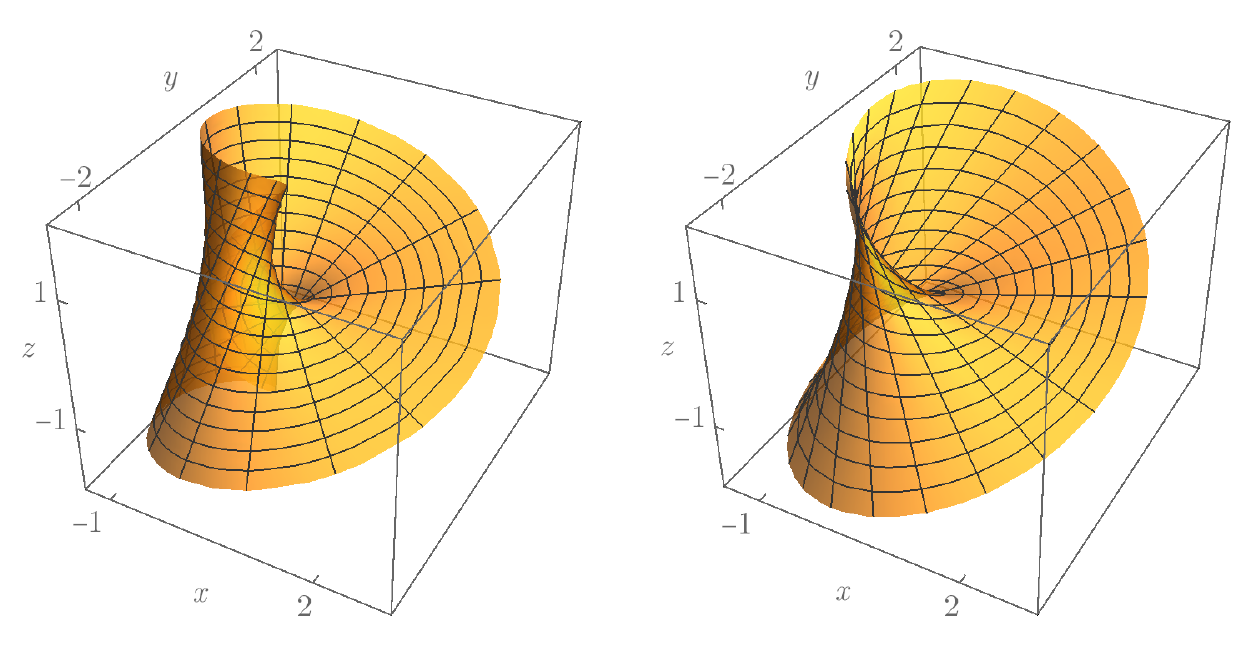}
	\caption{Two realizations of the M\"obius strip: the common one (left panel) and the simple one (right panel), both with 
$\de=1.97$.}  
\end{figure}

Therefore, let us prepare appropriate parts of the surface for better viewing of the self-intersection set $I_\de$ and its neighborhood. 
Toward that end, suppose that $\de>\sqrt 2$, for the self-intersection set $I_\de$ to be nonempty.  
Next, let 
\begin{equation*}
	h_2:=h_2(\de):=\arcsin\tfrac{\de\wedge2}2\in\big(\tfrac\pi4,\tfrac\pi2\big]\quad\text{and}\quad h_3:=h_3(\de):=\pi-h_2\in\big[\tfrac\pi2,\tfrac{3\pi}4\big). 
\end{equation*}
Consider then the following parts of the image $\w(P_\de)$ of the M\"obius strip $P_\de$ under map $\w$: 
\begin{equation*}
\begin{aligned}
	S_{1,\mathsf{bot}}&:=\big\{\w(t_1,r)\colon t_1\in[\tfrac\pi2,2h_2],\,
	r\in[-2\cos\tfrac\pi4,-2\cos h_2]\big\}, \\ 
	S_{1,\mathsf{top}}&:=\big\{\w(t_1,r)\colon t_1\in[2h_3,\tfrac{3\pi}2],\,
	r\in[-2\cos h_3,-2\cos\tfrac{3\pi}4]\big\}, \\ 
	S_{2,\mathsf{bot}}&:=\big\{\w(t_2,r)\colon t_2\in[\pi-2h_2,\tfrac\pi2],\,
	r\in[-2\sin h_2,-2\sin\tfrac\pi4]\big\}, \\ 	
	S_{2,\mathsf{top}}&:=\big\{\w(t_2,r)\colon t_2\in[\tfrac{3\pi}2,3\pi-2h_3),\,
	r\in[2\sin\tfrac{3\pi}4,2\sin h_3]\big\}. 	
\end{aligned}	
\end{equation*}
Concerning the ranges of $t_1$, $t_2$, and $r$ in the above display, note that 
\begin{align*} 
[\tfrac\pi2,2h_2]
	\times[-2\cos\tfrac\pi4,-2\cos h_2]
	& \subset(0,\pi]\times[-\sqrt2,0]\subset[0,2\pi)\times[-\de,\de]=P_\de, \\ 
[2h_3,\tfrac{3\pi}2]
	\times[-2\cos h_3,-2\cos\tfrac{3\pi}4]
	& \subset[\pi,2\pi)\times[0,\sqrt2]\subset[0,2\pi)\times[-\de,\de]=P_\de, \\ 
[\pi-2h_2,\tfrac\pi2]
	\times[-2\sin h_2,-2\sin\tfrac\pi4]
	& \subset[0,\pi)\times[-\de,-\sqrt2]\subset[0,2\pi)\times[-\de,\de]=P_\de, \\ 	
[\tfrac{3\pi}2,3\pi-2h_3)
	\times[2\sin\tfrac{3\pi}4,2\sin h_3]
	& \subset(\pi,2\pi)\times[\sqrt2,\de]\subset[0,2\pi)\times[-\de,\de]=P_\de.   	
\end{align*}	
So, each of the surfaces $S_{1,\mathsf{bot}},S_{1,\mathsf{top}},S_{2,\mathsf{bot}},S_{2,\mathsf{top}}$ is indeed a part of the image $\w(P_\de)$ of the M\"obius strip $P_\de$ under map $\w$. 

Now we are ready to illustrate Proposition~\ref{prop:new}, which is done in Figure~2. 

\begin{figure}[h]
	\centering
		\includegraphics[width=1\textwidth]{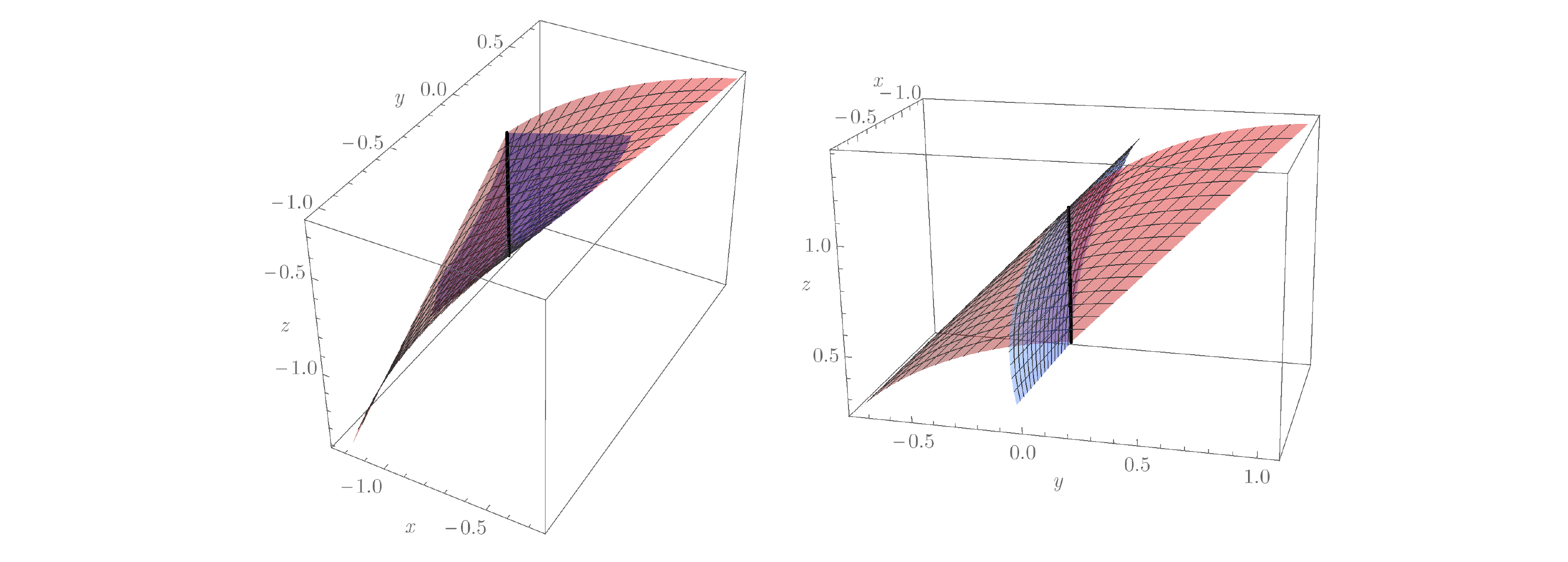}
	\label{fig:self-intersect}
	\caption{Parts of the image $\w(P_\de)$ of the M\"obius strip $P_\de$ under map $\w$ with $\de=1.97$. 
	Left panel: surfaces $S_{1,\mathsf{bot}}$ (red) and $S_{2,\mathsf{bot}}$ (blue), and the bottom half $\{(-1,0,s)\colon s_\de\le-s<1\}$ (thick, black) of the self-intersection set $I_\de$. 
	Right panel: surfaces $S_{1,\mathsf{top}}$ (red) and $S_{2,\mathsf{top}}$ (blue), and the top half $\{(-1,0,s)\colon s_\de\le s<1\}$ (thick, black) of the self-intersection set $I_\de$. 
	Here, $s(\de)=s(1.97)\approx0.34$. 
	}  
\end{figure}

In view of the above results and, in particular, Figure~2, 
it may seem unlikely that the simple M\"obius strip realization is the closure of the graph of a rational function on a polynomially defined subset of $\R^2$. Yet, we have 



\begin{theorem}\label{th:xyz}
For any real $\de>0$, the image $\w(P_\de)$ of the M\"obius strip $P_\de$ under the map $\w$ coincides with the topological closure of the graph of a rational function $f\colon S_\de\to\R$ on a polynomially defined set $S_\de$, where 
\begin{equation}\label{eq:f}
	f(x,y):=y-\frac{2(x+1)y}{(x+1)^2+y^2}=y\,\frac{x^2+y^2-1}{(x+1)^2+y^2}
\end{equation}
for all $(x,y)\in S_\de$,  
\begin{equation}\label{eq:g}
	S_\de:=\big\{(x,y)\in\R^2\setminus\{(-1,0)\}\colon g(x,y)\le\de^2\big\},\quad\text{and}\quad
	g(x,y):=\frac{(x^2 + y^2 - 1)^2}{(x + 1)^2 + y^2}.  
\end{equation}
\end{theorem}

Clearly, the set $S_\de$ is increasing in $\de$, and $\bigcup_{\de>0}S_\de=\R^2\setminus\{(-1,0)\}$. In Figure~\ref{fig:S_de}, one can see the sets $S_\de$ with $\de\in\{1/2, 1, \sqrt2, 2, 3\}$. 

\begin{figure}[h]%
\includegraphics[width=.5\textwidth]{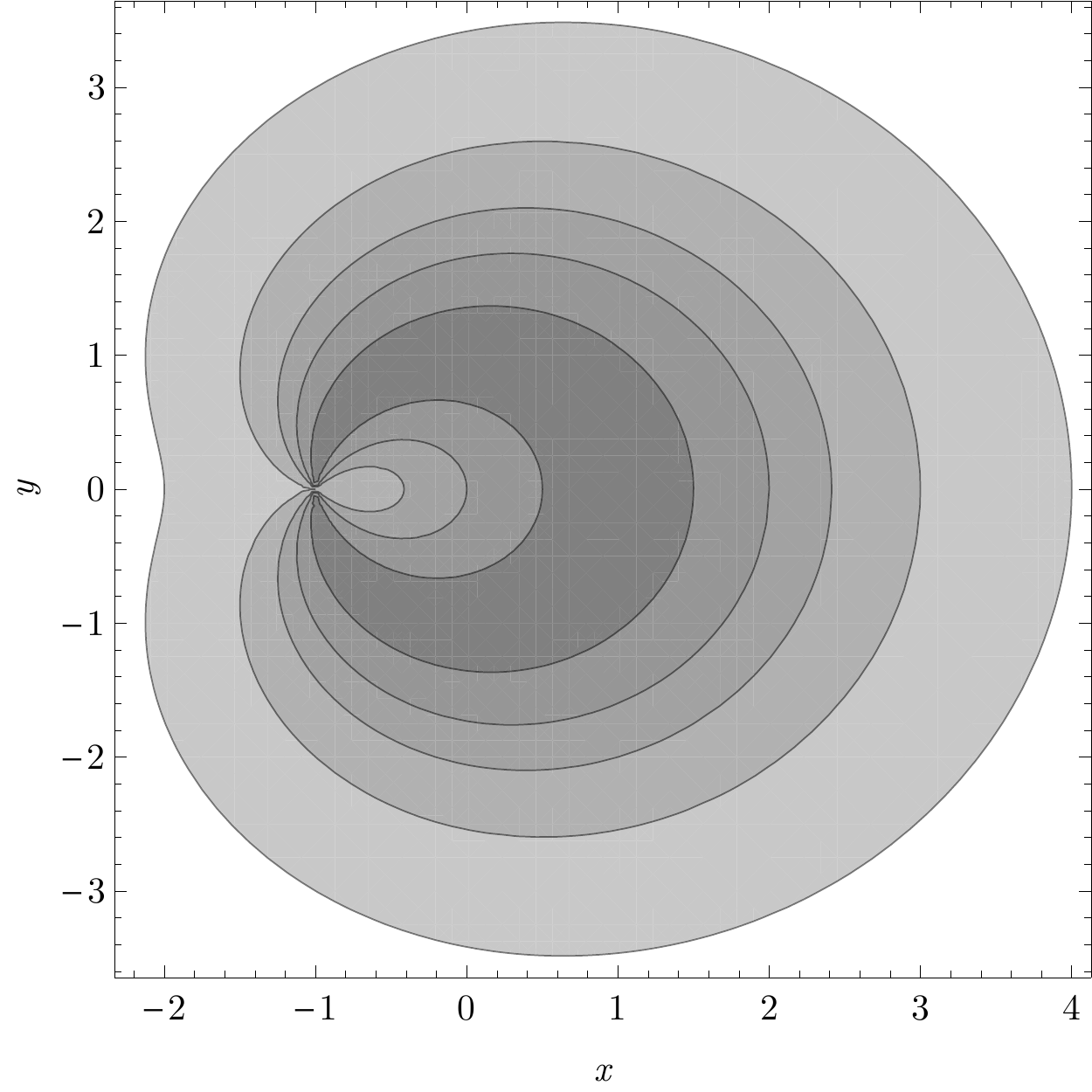}%
\caption{The increasing sets $S_\de$ for $\de\in\{\tfrac12, 1, \sqrt2, 2, 3\}$, from the darkest for $\de=\tfrac12$ to the lightest for $\de=3$.}%
\label{fig:S_de}%
\end{figure}

The proof of Theorem~\ref{th:xyz} is based on 

\begin{proposition}\label{prop:xyz}
One has 
\begin{equation*}
	\w(P_\de)\setminus\ell=G(f), 
\end{equation*}
where $\ell:=\{(-1,0,z)\colon z\in\R\}$ is the vertical line through point $(-1,0,0)$ in $\R^3$ and 
\begin{equation}\label{eq:G}
	G(f):=\big\{\big(x,y,f(x,y)\big)\colon(x,y)\in S_\de\big\}
\end{equation}
is the graph of the function $f$. 
\end{proposition}

Proposition~\ref{prop:xyz} is complemented by 


\begin{proposition}\label{prop:ell}
One has 
\begin{equation}\label{eq:J}
	\w(P_\de)\cap\ell=J_\de:=\{(-1,0,z)\colon|z|\le\si_\de\},  
\end{equation}
where 
\begin{equation}\label{eq:si}
	\si_\de:=2b_\de\sqrt{1-b_\de^2}\quad\text{and}\quad b_\de:=\frac{\de\wedge\sqrt2}2. 
\end{equation}
\end{proposition}

Clearly, it follows from \eqref{eq:I} and \eqref{eq:I=} that $I_\de\subseteq\w(P_\de)\cap\ell=J_\de$. Here one may also note that $\si_\de=1$ if $\de\ge\sqrt2$. 

Propositions~\ref{prop:new} and \ref{prop:ell} are illustrated in Figure~
5. 

\begin{figure}[h]
	\centering
		\includegraphics[width=.75\textwidth]{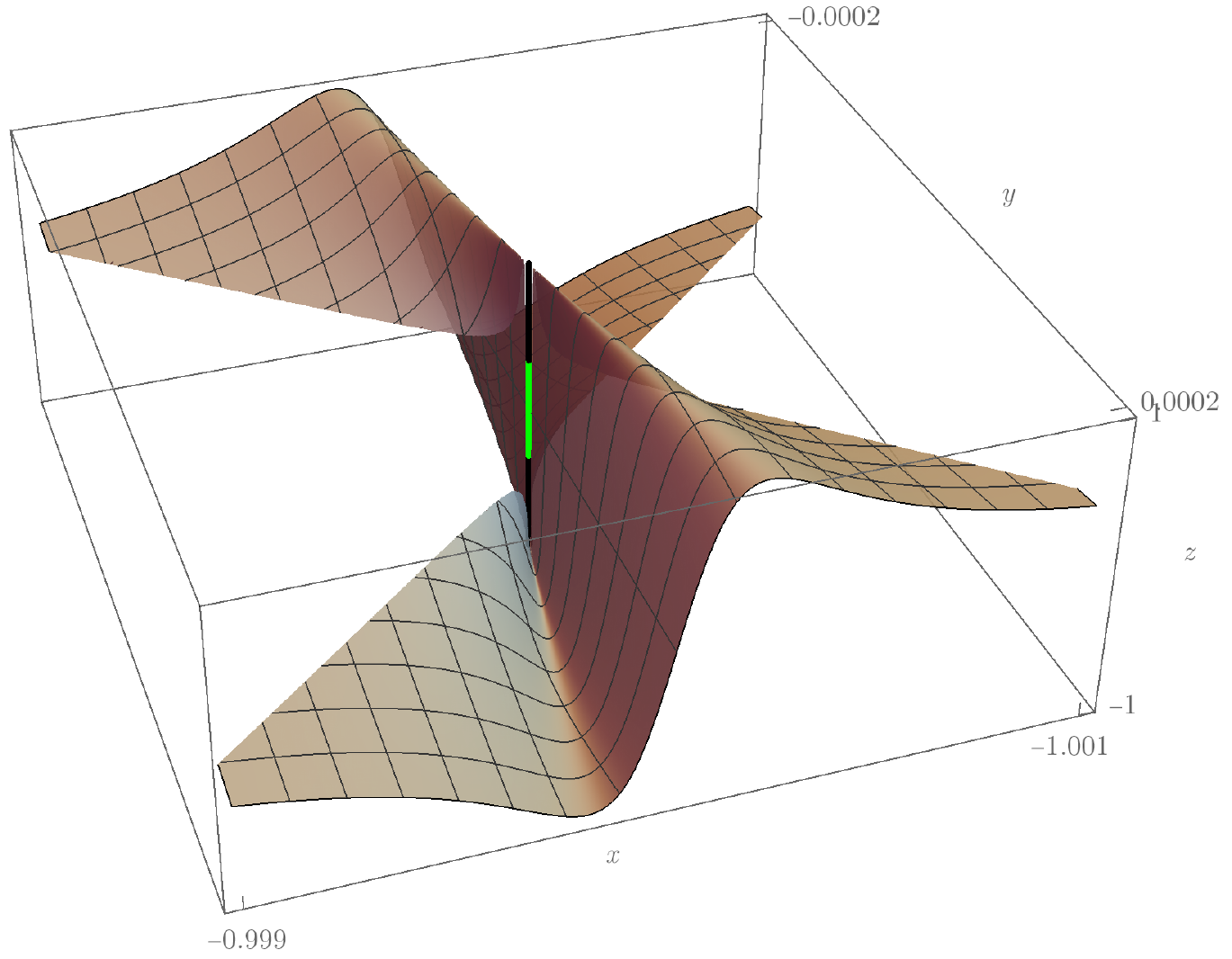}
	\label{fig:xyz}
	\caption{The part of the simple realization of the M\"obius strip with $\de=1.97$ in the neighborhood $[-1-\vp,-1+\vp]\times[-0.2\vp,0.2\vp]\times[-1-5\vp,1+5\vp]$ of the set $J_\de$, with $\vp=0.001$. Also shown are the sets $I_\de$ (the union of the two black vertical segments) and $J_\de\setminus I_\de$ (the green vertical segment); here $s_\de\approx0.34$ and $\si_\de=1$, so that $J_\de$ is the entire vertical segment from point $(-1,0,-1)$ to point $(-1,0,1)$.}  
\end{figure}

Immediately from Propositions~\ref{prop:xyz} and \ref{prop:ell}, we obtain the following complete description of $\w(P_\de)$: 

\begin{corollary}\label{cor:w(P)}
$\w(P_\de)=G(f)\cup\{(-1,0,z)\colon|z|\le\si_\de\}$. 
\end{corollary}

Thus, the simple realization $\w(P_\de)$ of the M\"obius strip is the union of a vertical segment with the graph of a rational function on a polynomially defined subset of $\R^2$. 
One should probably emphasize here that this conclusion holds for any, however large $\de>0$. 
To state the following corollary, let us consider the infinite-width M\"obius strip and its $\w$-image: 
\begin{equation*}
	P_\infty:=\bigcup_{\de>0}P_\de=[0,2\pi)\times\R\quad\text{and}\quad 
	\w(P_\infty)=\bigcup_{\de>0}\w(P_\de).
\end{equation*}
For each $(x,y)\in\R^2$, consider also the corresponding vertical ``cross-section'' of the ``entire'' simple M\"obius strip and its ``cardinality'':   
\begin{equation}\label{eq:Z,n}
	Z_\w(x,y):=\{z\in\R\colon(x,y,z)\in\w(P_\infty)\} \quad\text{and}\quad 
	n_\w(x,y):=\card Z_\w(x,y), 
\end{equation}
where $\card S$ equals the cardinality of a set $S$ if $S$ is finite and equals $\infty$ otherwise. 
Then Corollary~\ref{cor:w(P)} immediately yields 

\begin{corollary}\label{cor:n,Z}
$Z_\w(-1,0)=[-1,1]$ and $n_\w(x,y)=1$ for all $(x,y)\in\R^2\setminus\{(-1,0)\}$. 
\end{corollary}

\begin{center}
	***
\end{center}

Whereas the common M\"obius strip realization $\vv(P_\de)$ looks similar to $\w(P_\de)$ for small enough half-width $\de$, the situation changes dramatically when $\de$ is large enough. 
%
It is known and straightforward to check that for each real $\de>0$ the common M\"obius strip realization $\vv(P_\de)$ is a part of the cubic surface given by the ``implicit'' equation 
\begin{equation*}
	-y + x^2 y + y^3 - 2 
	x z - 2 x^2 z - 2 y^2 z + y z^2 = 0; 
\end{equation*}
see e.g.\ \cite{wolfram_moebius}. Since the left-hand side of this equation is a quadratic polynomial in $z$, one may expect that, in contrast with the simple realization $\w(P_\de)$ of the M\"obius strip, the common realization $\vv(P_\de)$ is ``mainly'' the union of the graphs of two algebraic (but not rational) functions. 
Proposition~\ref{prop:n,Z_v} below confirms and details such an expectation. 
 
Define $Z_\vv(x,y)$ and $n_\vv(x,y)$ the same way $Z_\w(x,y)$ and $n_\w(x,y)$ were defined in \eqref{eq:Z,n}, but using $\vv$ instead of $\w$. 
We have the following counterpart of Corollary~\ref{cor:n,Z}: 
 
\begin{proposition}\label{prop:n,Z_v}
For any $(x,y)\in\R^2$, 
\begin{equation}\label{eq:n=}
	n_\vv(x,y)=
	\left\{
	\begin{aligned}
	2&\text{\quad if\quad}x\ne-1\ \&\ y\ne0, \\ 
	1&\text{\quad if\quad}x=-1\ \&\ y\ne0, \\ 
	1&\text{\quad if\quad}x\notin\{-1,0\}\ \&\ y=0, \\ 	
	\infty&\text{\quad if\quad}x\in\{-1,0\}\ \&\ y=0.  	
	\end{aligned}
	\right.
\end{equation}
Moreover, $Z_\vv(-1,0)=Z_\vv(0,0)=\R$. 
\end{proposition} 

So, the ``typical'' vertical ``cross-section'' of the ``entire'' common M\"obius strip consists of precisely two points. 
Moreover, going over the lines of the proof of Proposition~\ref{prop:n,Z_v}, one can see that the set of all points $(x,y)\in\R^2$ for which the vertical ``cross-section'' of the common M\"obius strip $\vv(P_\de)$ contains exactly two distinct points is of positive Lebesgue measure if $\de>\sqrt2$ (but not if $\de\le\sqrt2$); in particular, look at \eqref{eq:r1,r2,rho} and note that 
\begin{equation*}
	\inf_{\th\in(0,\pi)}
	\Big[\Big(\dfrac{\rho-1}{\cos\tfrac\th2}\Big)^2\,\vee\Big(\dfrac{\rho+1}{\sin\tfrac\th2}\Big)^2\Big]
	=
	(\rho-1)^2+(\rho+1)^2
	=2+2\rho^2, 
\end{equation*}
which is greater than $2$ for any $\rho>0$, but is however close to $2$ if $\rho$ is small enough.

For the common M\"obius strip realization, we have the following counterpart of Theorem~\ref{th:new}: 

\begin{theorem}\label{th:old}
The map $P_\de\ni(t,r)\mapsto\vv(t,r)\in\vv(P_\de)$ is a homeomorphism 
iff $\de<2$. 
\end{theorem}

In \cite{schwarz90,wiki_moebius}, the common M\"obius strip realization is given with $\de\le1$, which may be compared with Theorem~\ref{th:old}. 

Theorem~\ref{th:old} and Proposition~\ref{prop:n,Z_v} are illustrated in Figure~6. 

\begin{figure}[h]
	\centering
		\includegraphics[width=.80\textwidth]
		{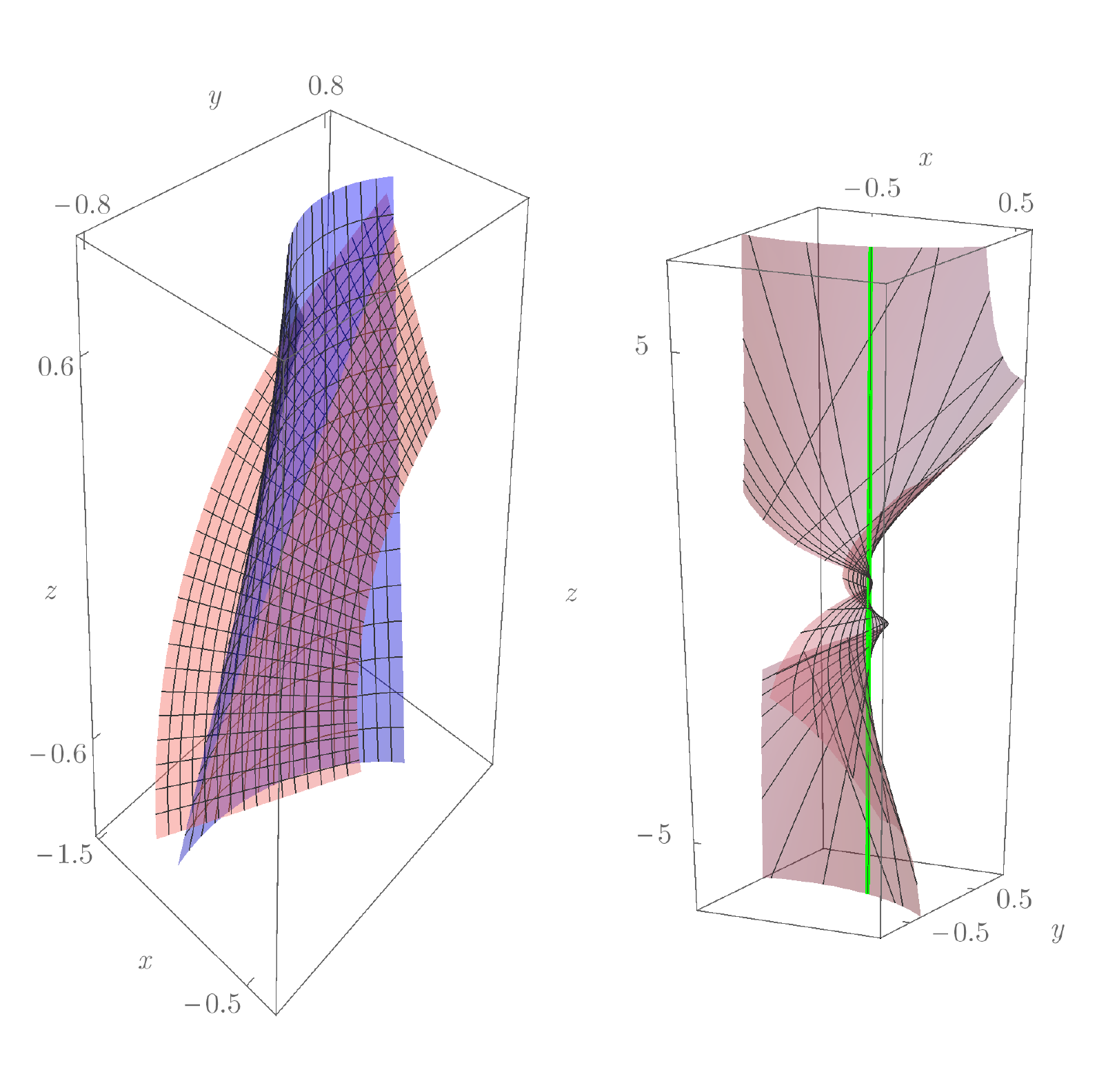}
	\label{fig:old-self}
	\caption{Parts of the image $\vv(P_\infty)$ of the M\"obius strip $P_\infty$ under map $\vv$. 
	Left panel: Surfaces $\vv\big([0,\frac\pi5]\times[-2.5,-1.5]\big)$ (pink), \break  $\vv\big([2\pi-\frac\pi5,2\pi]\times[1.5,2.5]\big)$ (pink), and $\vv\big([\pi-\frac\pi5,\pi+\frac\pi5]\times[-0.9,0.9]\big)$ (blue) --  in a neighborhood of point $(-1,0,0)$. Right panel: Surface $\big\{\vv\big(t,-\tfrac1{\cos(\tfrac t2+h)}\big)\colon t\in[0,2\pi],\,|h|\le0.5\big\}$ (pink) and a segment (green) of the vertical line $\ell$ through point $(0,0,0)$ --  in a neighborhood of point $(0,0,0)$. 	
	}  
\end{figure}

Significant efforts have been directed (see e.g.\ \cite{halp-weaver,schwarz90,starostin-heijden}) towards obtaining physically realizable -- that is, requiring the least possible bending energy and no stretching/shrinking -- models of the M\"obius strip to be made from a paper band, in a sense culminating in \cite{starostin-heijden}. However, the corresponding equations in those papers are much more complicated than \eqref{eq:new} and even \eqref{eq:old}; only numerical solutions were given in \cite{starostin-heijden} for the physically realizable model of a paper M\"obius strip presented there.  

\section{Proofs}\label{proofs} 
Modulo Proposition~\ref{prop:new}, the following is easy: 
\begin{proof}[Proof of Theorem~\ref{th:new}]
It is not hard to see that the M\"obius strip $P_\de$ is compact in its topology. Also, since $\w(2\pi-,r)=\w(0,-r)$ for $r\in[-\de,\de]$, it is easy to check that the map $\w$ is continuous. It follows that $\w$ is homeomorphic iff it is one-to-one. It remains to refer to \eqref{eq:I=empty}. 
\end{proof}

We still have to provide 

\begin{proof}[Proof of Proposition~\ref{prop:new}]
Suppose that $(p_1,p_2)\in Q_\de$, so that 
$(t_1,r_1)$ and $(t_2,r_2)$ are distinct points in $P_\de$ such that $\w(t_1,r_1)=\w(t_2,r_2)$. 
Then, by \eqref{eq:new}, $\sin t_1+r_1\sin\tfrac{t_1}2=\sin t_2+r_2\sin\tfrac{t_2}2$ and 
$r_1\sin\tfrac{t_1}2=r_2\sin\tfrac{t_2}2$, whence  $\sin t_1=\sin t_2$. Because the points $(t_1,r_1)$ and $(t_2,r_2)$ are in $P_\de$ and hence $t_1$ and $t_2$ are in $[0,2\pi)$, it follows that either $t_1=t_2$ or 
$t_1+t_2=(2k+1)\pi$ for some $k\in\{0,1\}$. 

If $t_1=t_2=:t$, then \eqref{eq:new} implies $\cos t+r_1\cos\tfrac t2=\cos t+r_2\cos\tfrac t2$ and $r_1\sin\tfrac t2=r_2\sin\tfrac t2$. Since $\cos^2\tfrac t2+\sin^2\tfrac t2=1\ne0$, we have $(t_1,r_1)=(t_2,r_2)$, which is a contradiction. 

So, $t_1\ne t_2$ and $t_1+t_2=(2k+1)\pi$ for some $k\in\{0,1\}$, whence $t_1+t_2\in\{\pi,3\pi\}$.  
In particular, this verifies the first condition in \eqref{eq:conds}. 
It also follows that  
\begin{equation}\label{eq:t2 to t1}
	\cos\tfrac{t_2}2=(-1)^k\sin\tfrac{t_1}2\quad\text{and}\quad \sin\tfrac{t_2}2=(-1)^k\cos\tfrac{t_1}2. 
\end{equation} 
Therefore, the assumption $\w(t_1,r_1)=\w(t_2,r_2)$ implies the following system of linear equations:  
\begin{equation}\label{eq:syst}
	r_1\cos \tfrac{t_1}2+2\cos^2\tfrac{t_1}2=(-1)^k r_2\sin \tfrac{t_1}2+2\sin^2\tfrac{t_1}2\quad\text{and}\quad 
	r_1\sin \tfrac{t_1}2=(-1)^k r_2\cos \tfrac{t_1}2  
\end{equation}
for $r_1$ and $r_2$.  
The determinant of this system is $(-1)^k (\sin^2\tfrac{t_1}2-\cos^2\tfrac{t_1}2)=\break 
(-1)^{k+1}\cos t_1$. If it is zero, then $t_1\in\{\pi/2,3\pi/2\}$; hence, in view of the conditions $t_1+t_2\in\{\pi,3\pi\}$ and $t_1\in[0,2\pi)$ and $t_2\in[0,2\pi)$, we have either $t_1=\pi/2=t_2$ or $t_1=3\pi/2=t_2$, which contradicts the conclusion that $t_1\ne t_2$.  
So, 
\begin{equation}\label{eq:t1 notin}
	t_1\notin\{\pi/2,3\pi/2\},
\end{equation}
the determinant is nonzero, and hence the system \eqref{eq:syst} has a unique solution -- which is given by the formula 
\begin{equation}\label{eq:r1,r2=}
	r_1=
	-2\cos \tfrac{t_1}2\quad\text{and}\quad r_2=
	(-1)^{k+1}2\sin \tfrac{t_1}2, 
\end{equation}
for each $k\in\{0,1\}$. 
This verifies the second and third conditions in \eqref{eq:conds}. 

Moreover, because the points $(t_1,r_1)$ and $(t_2,r_2)$ are in $P_\de$, we have $|r_1|\vee|r_2|\le\de$. That is, in view of \eqref{eq:r1,r2=}, \eqref{eq:t1 notin}, and \eqref{eq:c,s},  
\begin{equation}\label{eq:T}
\begin{aligned}
	t_1\in T_\de:=&\{t\in[0,2\pi)\colon\sqrt2<|2\cos\tfrac t2|\vee|2\sin\tfrac t2|\le\de\} \\ 
	=&\{t\in[0,2\pi)\colon s_\de\le|\sin t|<1\},  
\end{aligned}	
\end{equation}
which completes the verification of all the four conditions in \eqref{eq:conds} -- under the assumption that $(p_1,p_2)\in Q_\de$. 

Vice versa, assume now that the conditions in \eqref{eq:conds} hold for some points $p_1=(t_1,r_1)$ and $p_2=(t_2,r_2)$ in 
$[0,2\pi)\times\R$ and some $k\in\{0,1\}$, and let us then show that $(p_1,p_2)\in Q_\de$. This is easy to do, as the above reasoning in this proof is essentially reversible. 
Anyway, recall here that the first condition in \eqref{eq:conds} implies \eqref{eq:t2 to t1}. So, it follows from \eqref{eq:new} 
and the first three conditions in \eqref{eq:conds} that 
\begin{equation}\label{eq:same}
	\w(t_1,r_1)=\w(t_2,r_2)=(-1,0,-\sin t_1),     
\end{equation}
which verifies the last condition in the definition \eqref{eq:Q} of $Q_\de$. 

The last condition in \eqref{eq:conds} implies that $t_1\notin\{\pi/2,3\pi/2\}$. Hence, by the first condition in \eqref{eq:conds}, $t_1\ne t_2$, which verifies the condition $p_1\ne p_2$ in the definition \eqref{eq:Q} of $Q_\de$. 

The last three conditions in \eqref{eq:conds} together with 
the second line in \eqref{eq:T} imply that $|r_1|\vee|r_2|\le\de$, which verifies the condition $(p_1,p_2)\in P_\de^2$ in the definition \eqref{eq:Q} of $Q_\de$. 
Thus, the first part of Proposition~\ref{prop:new} is proved. 

It remains to prove \eqref{eq:I=}. Take any $(x,y,z)\in I_\de$. Then, by \eqref{eq:I}, there exists some $(p_1,p_2)\in Q_\de$ such that $\w(p_1)=(x,y,z)$. So, by the just proved first part of Proposition~\ref{prop:new}, for some points $p_1=(t_1,r_1)$ and $p_2=(t_2,r_2)$ in 
$[0,2\pi)\times\R$ the conditions in \eqref{eq:conds} hold. So, in view of \eqref{eq:same}, 
$(x,y,z)=(-1,0,s)$ for $s:=-\sin t_1$, and then also $s_\de\le|s|<1$, so that $(x,y,z)$ belongs to the set on the right-hand side of equality \eqref{eq:I=}. 

Vice versa, take any $s$ such that $s_\de\le|s|<1$. To complete the proof of Proposition~\ref{prop:new}, it now remains to show that $(-1,0,s)\in I_\de$. Clearly, the condition $s_\de\le|s|<1$ implies that there is some $t\in(0,2\pi)$ such that $s=-\sin t_1$, so that the last condition in \eqref{eq:conds} is satisfied. 
Next, there is some $k\in\{0,1\}$ such that $k\pi<t_1\le(k+1)\pi$. 
Further, take $t_2$, $r_1$, and $r_2$ according to the first three conditions in \eqref{eq:conds}. Then $0\le t_2<2\pi$ and, by the first part of Proposition~\ref{prop:new}, for the resulting points $p_1=(t_1,r_1)$ and $p_2=(t_2,r_2)$ one has $(p_1,p_2)\in Q_\de$. Moreover, again by \eqref{eq:same},  
indeed $
(-1,0,s)=(-1,0,-\sin t_1)\in I_\de$.  

Proposition~\ref{prop:new} is now completely proved. 
\end{proof}

Assuming that Proposition~\ref{prop:xyz} holds, it is easy to complete 

\begin{proof}[Proof of Theorem~\ref{th:xyz}]
The image $\w(P_\de)$ of the compact set $P_\de$ under the continuous map $\w$ is compact and hence closed. So, in view of Proposition~\ref{prop:xyz}, it suffices to show that the set $	\w(P_\de)\setminus\ell$ is dense in $\w(P_\de)$. 

To do this, take any $(t,r)\in P_\de$ such that $\w(t,r)\in\ell$, so that $\w(t,r)=(-1,0,z)$ for some $z\in\R$. We only need to show that this point, $\w(t,r)=(-1,0,z)$, can be approximated however closely by points in $\w(P_\de)\setminus\ell$. 

Consider first the case $r=0$. Then, by \eqref{eq:new}, 
the condition $\w(t,r)=(-1,0,z)$ implies that $\cos t=-1$ and hence $t=\pi$ and $\w(t,r)=(-1,0,0)$. Take now any $\vp\in(0,\pi)$. Then $(\pi+\vp,0)\in P_\de$ and $\w(\pi+\vp,0)=(-\cos\vp,-\sin\vp,0)\in\w(P_\de)\setminus\ell$. Moreover, letting $\vp\downarrow0$, we have $\w(\pi+\vp,0)\to(-1,0,0)=\w(t,r)$, thus obtaining the desired approximation. 

Assume now that $r\ne0$. Then for any $\vp\in(0,1)$ one has $(t,(1-\vp)r)\in P_\de$ and $\w(t,(1-\vp)r)=(-1-r\vp\cos\frac t2,-r\vp\sin\frac t2,z-r\vp\sin\frac t2)\in\w(P_\de)\setminus\ell$. 
Moreover, letting $\vp\downarrow0$, we have $\w(t,(1-\vp)r)\to(-1,0,z)=\w(t,r)$, thus obtaining the desired approximation in the latter case as well. 
\end{proof}

We still have to provide 

\begin{proof}[Proof of Proposition~\ref{prop:xyz}]
%
Let us first show that $\w(P_\de)\setminus\ell\subseteq G(f)$. To do that, take any $(x,y,z)\in\w(P_\de)\setminus\ell$, so that $(x,y)\ne(-1,0)$ and \eqref{eq:new} holds 
for some $(t,r)\in P_\de=[0,2\pi)\times[-\de,\de]$. We want then to show that $(x,y)\in S_\de$ and $f(x,y)=z$. 

By \eqref{eq:new}, 
\begin{equation}\label{eq:x...-y...}
	x\sin\tfrac t2-y\cos\tfrac t2=\cos t\sin\tfrac t2-\sin t\cos\tfrac t2=-\sin\tfrac t2. 
\end{equation}
Also, the condition $(t,r)\in P_\de$ implies that $0\le\tfrac t2<\pi$. 
So, if $\sin\tfrac t2=0$, then $t=0$ and hence, by \eqref{eq:new}, \eqref{eq:f}, and \eqref{eq:g}, 
$(x,y,z)=\w(t,r)=(1+r,0,0)$, $f(x,y)=0=z$, and $g(x,y)=r^2\le\de^2$. Thus, $(x,y)\in S_\de$ and $f(x,y)=z$ -- in the case when $\sin\tfrac t2=0$. 

Consider now the case when $\sin\tfrac t2\ne0$. Then \eqref{eq:x...-y...} yields $x+1=y\cot\tfrac t2$. So, $y=0$ would imply $x=-1$, in contradiction with the condition $(x,y)\ne(-1,0)$. 
It follows that $y\ne0$. Since $0\le\tfrac t2<\pi$, we have 
\begin{equation}\label{eq:cot,...}
	\cot\tfrac t2=\tfrac{x+1}y,\quad t=2\arccot\tfrac{x+1}y,\quad \sin t=\frac{2y(x+1)}{(x+1)^2+y^2};  
\end{equation}
here we use the 
convention that the values of the function $\arccot$ are in the interval $(0,\pi)$. 
Now, by \eqref{eq:new} and \eqref{eq:f}, 
\begin{equation}\label{eq:z=f(x,y)}
	z=y-\sin t=y-\frac{2y(x+1)}{(x+1)^2+y^2}=f(x,y). 
\end{equation}
Moreover, the conditions $(t,r)\in P_\de$, \eqref{eq:new}, $\sin\tfrac t2\ne0$, \eqref{eq:cot,...}, \eqref{eq:z=f(x,y)}, \eqref{eq:f}, and \eqref{eq:g} yield
\begin{multline*}
	\de^2\ge r^2=\frac{z^2}{\sin^2\tfrac t2}
	=z^2(1+\cot^2\tfrac t2)=z^2\,\frac{(x+1)^2+y^2}{y^2}
	=f(x,y)^2\,\frac{(x+1)^2+y^2}{y^2} \\ 
	=g(x,y), 
\end{multline*}
so that $(x,y)\in S_\de$. 
Thus, in the case $\sin\tfrac t2\ne0$ as well, we have $(x,y)\in S_\de$ and $f(x,y)=z$. 
This proves that $\w(P_\de)\setminus\ell\subseteq G(f)$. 

It remains to prove that $G(f)\subseteq \w(P_\de)\setminus\ell$. To do this, take any $(x,y,z)\in G(f)$, so that $(x,y)\in S_\de$ and $f(x,y)=z$. In view of \eqref{eq:G} and \eqref{eq:g}, $G(f)\cap\ell=\emptyset$. 
So, it suffices to show that \eqref{eq:new} holds 
for some $(t,r)\in P_\de=[0,2\pi)\times[-\de,\de]$. 

If $y=0$, then $z=f(x,0)=0$, and the condition $(x,y)\in S_\de$ implies $(x-1)^2=g(x,y)\le\de^2$ and 
\eqref{eq:new} obviously holds 
for $(t,r)=(0,x-1)\in P_\de=[0,2\pi)\times[-\de,\de]$. 

Consider now the case $y\ne0$. Let then 
\begin{equation*}
t:=2\arccot\frac{x+1}y
\quad\text{and}\quad r:=\frac{y-\sin t}{\sin\frac t2}
=\frac y{\sin\frac t2}-2\cos\tfrac t2,   	
\end{equation*}
whence 
\begin{gather*}
t\in(0,2\pi),\quad 
\sin\tfrac t2=\frac{|y|}{\sqrt{(x+1)^2+y^2}}>0,\quad 
\cos\tfrac t2=\frac{(x+1)\sign y}{\sqrt{(x+1)^2+y^2}},\\  
\sin t=\frac{2(x+1)y}{(x+1)^2+y^2}, \quad \cos t=\frac{(x+1)^2-y^2}{(x+1)^2+y^2},\\  	
\quad \cos t+r\cos\tfrac t2=x,\quad 
\sin t+r\sin\tfrac t2=y,\quad 
r^2=g(x,y)^2\le\de^2. 
\end{gather*}
So, $(t,r)\in P_\de$. Moreover, it follows from \eqref{eq:new}, the last line of the above display,  the condition $y\ne0$, and the already proved set inclusion $\w(P_\de)\setminus\ell\subseteq G(f)$ that for some $\tilde z\in\R$ one has 
$(x,y,\tilde z)=\w(t,r)\in\w(P_\de)\setminus\ell\subseteq G(f)$. 
But it was assumed that $(x,y,z)\in G(f)$, and $f$ is a function. We conclude that $\tilde z=z$, whence $\w(t,r)=(x,y,z)
$; that is, indeed \eqref{eq:new} holds 
for the point $(t,r)\in P_\de$ constructed above.  
\end{proof}

\begin{proof}[Proof of Proposition~\ref{prop:ell}]
Let us first show that $\w(P_\de)\cap\ell\subseteq J_\de$. Take any $(x,y,z)\in\w(P_\de)\cap\ell$, so that $x=-1$, $y=0$, and 
\eqref{eq:new} holds for some $(t,r)\in P_\de$. Hence, $t\in[0,2\pi)$, $|r|\le\de$, 
$-1=x=\cos t+r\cos\tfrac t2$, $0=y=\sin t+r\sin\tfrac t2$, and $z=r\sin\tfrac t2$. 

If $\sin\tfrac t2=0$, then $z=0$ and hence $(x,y,z)\in J_\de$. 

Suppose now that $\sin\tfrac t2\ne0$. Then $0=y=\sin t+r\sin\tfrac t2$ yields $r=-2\cos\tfrac t2$ and hence $\de\ge|r|=2|\cos\tfrac t2|$ and $z=r\sin\tfrac t2=-\sin t$. On the other hand, 
\begin{equation}\label{eq:iff t}
\big\{-\sin t\colon t\in[0,2\pi),\,2|\cos\tfrac t2|\le\de\big\}=[-\si_\de,\si_\de]. 
\end{equation}
So, $(x,y,z)\in J_\de$ in the case $\sin\tfrac t2\ne0$ as well, which verifies the set inclusion $\w(P_\de)\cap\ell\subseteq J_\de$. 

The above reasoning is essentially reversible. Indeed, take any $(x,y,z)\in J_\de$, so that $x=-1$, $y=0$, and $z\in[-\si_\de,\si_\de]$. By \eqref{eq:iff t}, we can find some $t\in[0,2\pi)$ such that $z=-\sin t$ and $2|\cos\tfrac t2|\le\de$. 
%
Letting now $r:=-2\cos\tfrac t2$, we have $|r|\le\de$, so that $(t,r)\in P_\de$; moreover, then  $(x,y,z)=(-1,0,-\sin t)=\w(t,r)$ by \eqref{eq:new}, so that $(x,y,z)\in w(P_\de)\cap\ell$. 
This verifies the set inclusion $J_\de\subseteq\w(P_\de)\cap\ell$ and thus completes the proof of Proposition~\ref{prop:ell}. 
\end{proof}

\begin{proof}[Proof of Proposition~\ref{prop:n,Z_v}]
Take any $(x,y)\in\R^2$. Let 
\begin{equation*}
\begin{aligned}
	P(x,y)&:=\big\{(t,r)\in P_\infty\colon(1+r\cos\tfrac t2)\cos t=x,\, (1+r\cos\tfrac t2)\sin t=y\big\}, \\  
	N(x,y)&:=\card P(x,y). 
\end{aligned}	
\end{equation*}
In view of \eqref{eq:old}, 
\begin{equation*}
	Z_\vv(x,y)=\big\{r\sin\tfrac t2\colon(t,r)\in P(x,y)\big\},  
\end{equation*}
whence 
\begin{equation}\label{eq:n le N}
	n_\vv(x,y)=\card Z_\vv(x,y)\le\card P(x,y)=N(x,y). 
\end{equation}

Consider also the pair $(\rho,\th)=\big(\rho(x,y),\th(x,y)\big)$ of the ``polar coordinates'' of the point $(x,y)$, which is uniquely determined by the conditions 
\begin{equation}\label{eq:polar}
	(x,y)=\rho(\cos\th,\sin\th)\quad\text{and}\quad (\rho,\th)\in\R\times[0,\pi),
\end{equation}
provided that $(x,y)\ne(0,0)$. The choice $[0,\pi)$ for the range of values of $\th$ will be convenient, because it implies that  
\begin{equation}\label{eq:cos th/2 ne0}
	\cos\tfrac\th 2\ne0. 
\end{equation}
 
Proposition~\ref{prop:n,Z_v} will easily follow from a few lemmas: 

\begin{lemma}\label{lem:0,0}
$Z_\vv(0,0)=\R$. 
\end{lemma}

\begin{proof}[Proof of Lemma~\ref{lem:0,0}]
Take any $z\in\R$ and let $t:=-2\arctan z+2\pi\ii{z>0}$, where $\ii{\cdot}$ denotes the indicator function. Then $t\in[0,2\pi)\setminus\{\pi\}$, so that $\cos\tfrac t2\ne0$. Let now $r:=-1/\cos\tfrac t2$. Then $\vv(t,r)=(0,0,z)$. It remains to recall the definition of $Z_\vv$ (cf.\ \eqref{eq:Z,n}). 
\end{proof}

\begin{lemma}\label{lem:-1,0}
$Z_\vv(-1,0)=\R$. 
\end{lemma}

\begin{proof}[Proof of Lemma~\ref{lem:-1,0}]
For all $z\in\R$ one has $\vv(\pi,z)=(-1,0,z)$. 
\end{proof}

\begin{lemma}\label{lem:n ge1}
$n_\vv(x,y)\ge1$. 
\end{lemma}

\begin{proof}[Proof of Lemma~\ref{lem:n ge1}]
Let $r:=(\rho-1)/\cos\tfrac\th2$; this definition makes sense, in view of \eqref{eq:cos th/2 ne0}. 
Then, by \eqref{eq:old}, $\vv(\th,r)=(x,y,z)$ 
for some $z\in\R$. Hence, $Z_\vv(x,y)\ne\emptyset$ and $n_\vv(x,y)\ge1$. 
\end{proof}


%

\begin{lemma}\label{lem:y=0}
Suppose that $x\notin\{-1,0\}$ and $y=0$. Then $n_\vv(x,y)=1$. 
\end{lemma}

\begin{proof}[Proof of Lemma~\ref{lem:y=0}]
Take any $(t,r)\in P(x,0)$, so that $(1+r\cos\tfrac t2)\cos t=x\ne0$ and $(1+r\cos\tfrac t2)\sin t=0$. Then $1+r\cos\tfrac t2\ne0$, $\sin t=0$, and $t\in\{0,\pi\}$. If $t=\pi$, then $x=-1$, a contradiction. So, $t=0$, $1+r=x$, and $r=x-1$. This means that $P(x,0)\subseteq\{(0,x-1)\}$ and $N(x,0)\le1$. It remains to recall that $n_\vv(x,y)\le N(x,y)$ (by \eqref{eq:n le N}) and refer to Lemma~\ref{lem:n ge1}. 
\end{proof}

\begin{lemma}\label{lem:y ne0}
Suppose that $y\ne0$. Then $n_\vv(x,y)=1$ if $x=-1$ and $n_\vv(x,y)=2$ if $x\ne-1$. 
\end{lemma}

\begin{proof}[Proof of Lemma~\ref{lem:y ne0}]
In view of \eqref{eq:old} and the condition $y\ne0$, for any $(t,r)\in P(x,y)$ one has $\sin t\ne0$ and hence $\cos\tfrac t2\ne0$, so that $r$ is uniquely determined, by the formula $r=\big(\tfrac y{\sin t}-1\big)/\cos\tfrac t2$, given $t$ and the condition $(t,r)\in P(x,y)$. 
Therefore, by \eqref{eq:polar}, the condition that two points in $P(x,y)$ are distinct means precisely that they are of the form $(\th,r_1)$ and $(\th+\pi,r_2)$ for some $\th\in[0,\pi)$ and real $r_1$ and $r_2$ such that $1+r_1\cos\tfrac\th2=\rho=-(1+r_2\cos\tfrac{\th+\pi}2)$ or, equivalently, 
\begin{equation}\label{eq:r1,r2,rho}
	r_1=\dfrac{\rho-1}{\cos\tfrac\th2}\quad\text{and}\quad r_2=\dfrac{\rho+1}{\sin\tfrac\th2}. 
\end{equation}
Here we have taken into account that $\cos\tfrac\th2\ne0$ (because $\th\in[0,\pi)$) and $\sin\tfrac\th2\ne0$ (because otherwise $\th=0$ and hence $y=\rho\sin\th=0$). So, here 
$P(x,y)=\Big\{\Big(\th,\dfrac{\rho-1}{\cos\tfrac\th2}\Big),\Big(\th+\pi,
\dfrac{\rho+1}{\sin\tfrac\th2}\Big)\Big\}$.  
Hence, in view of \eqref{eq:old} and \eqref{eq:r1,r2,rho}, the condition $n_\vv(x,y)=1$ means precisely that $\dfrac{\rho-1}{\cos\tfrac\th2}\,\sin\tfrac\th2=\dfrac{\rho+1}{\sin\tfrac\th2}\,\sin\tfrac{\th+\pi}2$, which can be rewritten as $\rho\cos\th=-1$ or simply as $x=-1$. 
It also follows that in the remaining case when $x\ne-1$ one has $n_\vv(x,y)=2$.  
\end{proof}

Now the first two lines in \eqref{eq:n=} follow by Lemma~\ref{lem:y ne0}, the third line there by Lemma~\ref{lem:y=0}. *, and the fourth line there by Lemmas~\ref{lem:-1,0} and \ref{lem:-1,0}. 
The second sentence in Proposition~\ref{prop:n,Z_v} and hence the fourth line in \eqref{eq:n=} follow by Lemmas~\ref{lem:-1,0} and \ref{lem:-1,0}. 

Proposition~\ref{prop:n,Z_v} is thus completely proved. 
%
\end{proof}

\begin{proof}[Proof of Theorem~\ref{th:old}]
As in the proof of Theorem~\ref{th:new}, here it is easy to see that the map $P_\de\ni(t,r)\mapsto\vv(t,r)\in\vv(P_\de)$ is a homeomorphism 
iff it is one-to-one. 

Next, suppose that $\de\ge2$. Then $(0,-2)$ and $(\pi,0)$ are distinct points in $P_\de$ and, 
in view of \eqref{eq:old}, $\vv(0,-2)=(-1,0,0)=\vv(\pi,0)$. So, the map $P_\de\ni(t,r)\mapsto\vv(t,r)\in\vv(P_\de)$ is not one-to-one if $\de\ge2$. 

Suppose now that $\de<2$. Take any points $(t_1,r_1)$ and $(t_2,r_2)$ such that $\vv(t_1,r_1)=\vv(t_2,r_2)=:(x,y,z)$. We have then to show that $(t_1,r_1)=(t_2,r_2)$.  

If $t_1=t_2=:t$, then, by \eqref{eq:old}, 
\begin{equation}\label{eq:t,r1,r2}
	0=\|\vv(t,r_1)-\vv(t,r_2)\|^2=(r_1-r_2)^2, 
\end{equation}
where $\|\cdot\|$ denotes the Euclidean norm on $\R^3$. Hence, $r_1=r_2$, and the desired conclusion $(t_1,r_1)=(t_2,r_2)$ follows, whenever $t_1=t_2$. 
So, w.l.o.g.\ $t_1\ne t_2$. 

Next, consider the case when $\cos t_1=0=\cos t_2$, so that w.l.o.g.\ $t_1=\tfrac\pi2$ and $t_2=\tfrac{3\pi}2$, whence, in view of \eqref{eq:old}, $[y=]1+r_1/\sqrt2=-1+r_2/\sqrt2$ and $[z=]r_1/\sqrt2=r_2/\sqrt2$, so that $r_1=r_2=:r$ and $1+r/\sqrt2=-1+r/\sqrt2$, which is a contradiction. 
Hence, the case when $\cos t_1=0=\cos t_2$ is impossible, given the other conditions. 

Further, consider the case when exactly one of the numbers $\cos t_1$ and $\cos t_2$ is $0$. Then w.l.o.g.\ $\cos t_1=0\ne\cos t_2$, which implies $t_1=\pi+\vp\tfrac\pi2$ for some $\vp\in\{-1,1\}$, $x=0$, $1+r_2\cos\tfrac{t_2}2=0$, $y=0$, $1-r_1\vp/\sqrt2=1+r_1\cos\tfrac{t_1}2=\vp y=0$, $r_1=\vp\sqrt2$, and  $z=r_1\sin\tfrac{t_1}2=\vp=r_2\sin\tfrac{t_2}2$. Further, equalities $1+r_2\cos\tfrac{t_2}2=0$ and $\vp=r_2\sin\tfrac{t_2}2$ yield $\tan\tfrac{t_2}2=-\vp$, whence $\cos t_2=0$, which contradicts the current case condition. 

Further yet, consider the case when $\cos t_1\ne0$ and $\cos t_2\ne0$ but $x=0$. Then $1+r_1\cos\tfrac{t_1}2=0=1+r_2\cos\tfrac{t_2}2$ and $r_1\sin\tfrac{t_1}2=[z=]r_2\sin\tfrac{t_2}2$, whence $r_1=-1/\cos\tfrac{t_1}2$, $r_2=-1/\cos\tfrac{t_2}2$, $\tan\tfrac{t_1}2=\tan\tfrac{t_2}2$, and $t_1=t_2$, which contradicts the assumption $t_1\ne t_2$. 

It remains to consider the case when $x\ne0$, so that $\cos t_1\ne0$ and $\cos t_2\ne0$. Then, by \eqref{eq:old}, $\tan t_1=\tfrac yx=\tan t_2$, so that $|t_1-t_2|=\pi$. So, w.l.o.g.\ $t_1=t$ and $t_2=t+\pi$ for some $t\in[0,\pi)$. Hence, 
$[x=](1+r_1\cos\tfrac t2)\cos t=(1-r_2\sin\tfrac t2)(-\cos t)$ and 
$[y=](1+r_1\cos\tfrac t2)\sin t=(1-r_2\sin\tfrac t2)(-\sin t)$. Since $\cos t$ and $\sin t$ do not vanish simultaneously, we have $1+r_1\cos\tfrac t2=-(1-r_2\sin\tfrac t2)$. Also, 
$r_1\sin\tfrac t2=[z=]r_2\cos\tfrac t2$. So, we have the system $r_1\cos\tfrac t2-r_2\sin\tfrac t2=-2$ and $r_1\sin\tfrac t2-r_2\cos\tfrac t2=0$ of linear equations for $r_1,r_2$. Its determinant is $-\cos^2\tfrac t2+\sin^2\tfrac t2=-\cos t\ne0$, and its only solution is given by $r_1=-2\cos\tfrac t2/\cos t$ and $r_2=-2\sin\tfrac t2/\cos t$, whence $\de^2\ge r_1^2\vee r_2^2\ge|r_1^2-r_2^2|=4/|\cos t|\ge4$, so that $\de\ge2$, which contradicts the assumption $\de<2$. 

This completes the proof of Theorem~\ref{th:old}. 
\end{proof}

\bibliographystyle{elsarticle-num} 

\bibliography{P:/pCloudSync/mtu_pCloud_02-02-17/bib_files/citations12.13.12}

\def\cprime{$'$} \def\polhk#1{\setbox0=\hbox{#1}{\ooalign{\hidewidth
  \lower1.5ex\hbox{`}\hidewidth\crcr\unhbox0}}}
  \def\polhk#1{\setbox0=\hbox{#1}{\ooalign{\hidewidth
  \lower1.5ex\hbox{`}\hidewidth\crcr\unhbox0}}}
  \def\polhk#1{\setbox0=\hbox{#1}{\ooalign{\hidewidth
  \lower1.5ex\hbox{`}\hidewidth\crcr\unhbox0}}} \def\cprime{$'$}
  \def\polhk#1{\setbox0=\hbox{#1}{\ooalign{\hidewidth
  \lower1.5ex\hbox{`}\hidewidth\crcr\unhbox0}}} \def\cprime{$'$}
  \def\polhk#1{\setbox0=\hbox{#1}{\ooalign{\hidewidth
  \lower1.5ex\hbox{`}\hidewidth\crcr\unhbox0}}} \def\cprime{$'$}
  \def\cprime{$'$}
\begin{thebibliography}{1}
\expandafter\ifx\csname url\endcsname\relax
  \def\url#1{\texttt{#1}}\fi
\expandafter\ifx\csname urlprefix\endcsname\relax\def\urlprefix{URL }\fi
\expandafter\ifx\csname href\endcsname\relax
  \def\href#1#2{#2} \def\path#1{#1}\fi

\bibitem{wolfram_moebius}
E.~W. Weisstein, {M}\"obius strip. {F}rom {M}ath{W}orld -- {A} {W}olfram {W}eb
  {R}esource, \url{http://mathworld.wolfram.com/MoebiusStrip.html} (2016).

\bibitem{schwarz90}
G.~E. Schwarz, \href{http://dx.doi.org/10.2307/2324325}{The dark side of the
  {M}oebius strip}, Amer. Math. Monthly 97~(10) (1990) 890--897.
\newline\urlprefix\url{http://dx.doi.org/10.2307/2324325}

\bibitem{wiki_moebius}
Wikipedia, {M}\"obius strip,
  \url{https://en.wikipedia.org/wiki/M%C3%B6bius_strip} (2016).

\bibitem{halp-weaver}
B.~Halpern, C.~Weaver, Inverting a cylinder through isometric immersions and
  isometric embeddings, Trans. Amer. Math. Soc. 230 (1977) 41--70.

\bibitem{starostin-heijden}
E.~L. Starostin, G.~H.~M. van~der Heijden, The shape of a {M}\"obius strip,
  Nature Materials 6 (2007) 563--567.

\end{thebibliography}


\end{document}